\documentclass[12pt]{article}
\usepackage{amsmath, amsthm, amscd, amsfonts, amssymb, graphicx}
\usepackage{amsmath,amssymb,latexsym}
\usepackage{mathtools}
\usepackage[T1]{fontenc}
\usepackage{graphicx,tikz,color,url}

\newtheorem{theorem}{Theorem}[section]

\newtheorem{lemma}[theorem]{Lemma}
\newtheorem{proposition}[theorem]{Proposition}

\newtheorem{corollary}[theorem]{Corollary}

\newcommand{\gp}{\mathop{\mathrm{gp}}}
\newcommand{\mut}{\mathop{\mu_\mathrm{t}}}
\newcommand{\muo}{\mathop{\mu_\mathrm{o}}}
\newcommand{\mud}{\mu_{\rm d}}

\newcommand{\gpt}{{\rm gp}_{\rm t}}
\newcommand{\gpd}{{\rm gp}_{\rm d}}
\newcommand{\gpo}{{\rm gp}_{\rm o}}

\title{Counting largest mutual-visibility and general position sets of glued $t$-ary trees}

\author{Dhanya Roy $^{a,}$\footnote{\tt dhanyaroyku@gmail.com, dhanyaroyku@cusat.ac.in} 
	\and Sandi Klav\v{z}ar $^{b,c,d,}$\footnote{corresponding author, available at sandi.klavzar@fmf.uni-lj.si}
	\and Aparna Lakshmanan S $^{a,}$\footnote{\tt aparnaren@gmail.com, aparnals@cusat.ac.in}\\\\
	$^{a}$ \small Department of Mathematics, 	Cochin University of Science and Technology, 
	\\ \small Cochin - 22, Kerala, India\\
	$^{b}$\small Faculty of Mathematics and Physics, University of Ljubljana, Slovenia\\
	$^{c}$ \small Institute of Mathematics, Physics and Mechanics, Ljubljana, Slovenia \\
	$^{d}$ \small Faculty of Natural Sciences and Mathematics, University of Maribor, Slovenia\\
}
\date{\today}

\begin{document}
\maketitle
	
\begin{abstract}
All four invariants of the mutual-visibility problem and, all four invariants of the general position problem are determined for glued binary trees. The number of the corresponding extremal sets is obtained in each of the eight situations. The results are further extended to glued $t$-ary trees, and some of them also to generalized glued binary trees. 
\end{abstract}
	
	
\medskip\noindent
{\bf AMS Subj.\ Class.\ (2020)}: 05C12, 05C69, 05C30

\section{Introduction}
	
General position and mutual-visibility are two fresh areas of metric and algorithmic graph theory. The concepts are complementary to each other, and together represent a flourishing area of research. 

Based on the motivation of robotic visibility, the graph mutual-visibility problem was introduced by Di Stefano~\cite{distefano-2022}. Given a set $S$ of vertices in a graph $G$, two vertices $u$ and $v$ are {\em mutually-visible} with respect to $S$, shortly {\em $S$-visible}, if there \underline{exists a shortest $u,v$-path} $P$ such that $V(P) \cap S \subseteq \{u,v\}$. The set $S$ is a {\em mutual-visibility set} if any two vertices from $S$ are $S$-visible. A largest mutual-visibility set of $G$ is a {\em $\mu$-set} and its size is the {\em mutual-visibility number} $\mu(G)$ of $G$. 

In~\cite{cicerone-2024}, the total mutual-visibility number was introduced, while the variety of mutual-visibility invariants was rounded off in~\cite{cicerone-2023a} by adding to the list the outer mutual-visibility number and the dual mutual-visibility number. A set $S\subseteq V(G)$ is an {\em outer mutual-visibility set} in $G$ if $S$ is a mutual-visibility set and every pair of vertices $u \in S$, $v \in V(G)\setminus S$ are $S$-visible. $S$ is a {\em dual mutual-visibility set} if $S$ is a mutual-visibility set and every pair of vertices $u, v \in V(G)\setminus S$ are $S$-visible. Finally, $S$ is a {\em total mutual-visibility set} if every pair of vertices in $G$ are $S$-visible. Largest outer/dual/total mutual-visibility sets are respectively called {\em $\muo$-sets}, {\em $\mud$-sets}, {\em $\mut$-sets}, their sizes being the {\em outer/dual/total mutual-visibility number} of $G$, respectively denoted by $\muo(G)$, $\mud(G)$, $\mut(G)$.

Given a set $S$ of vertices in a graph $G$, two vertices $u$ and $v$ are {\em $S$-positionable}, if \underline{for every shortest $u,v$-path} $P$ we have $V(P) \cap S \subseteq \{u,v\}$. (Note that if $uv\in E(G)$, then $u$ and $v$ are $S$-positionable.) Then $S$ is a {\em general position set}, if every $u, v \in S$ are $S$-positionable. A largest general position set is a {\em gp-set} and its size is the {\em general position number} $\gp(G)$ of $G$. These concepts were independently introduced in~\cite{chandran-2016} and in~\cite{manuel-2018}, where the latter paper is the origin of the terminology and notation we now use. We should add that the concept has been previously explored in the context of hypercubes~\cite{Korner-1995}. 

Following the pattern of mutual-visibility, the variety of general position invariants was presented in~\cite{TianKlavzar-2025}. The definition of the {\em outer/dual/total general position set} in $G$ is analogous, we just need to replace everywhere ``$S$-visible" by ``$S$-positionable." Largest corresponding sets are called {\em $\gpo$-sets}, {\em $\gpd$-sets}, {\em $\gpt$-sets} and their sizes are the {\em outer/dual/total general position number} of $G$, respectively denoted by $\gpo(G)$, $\gpd(G)$, $\gpt(G)$.

The literature on general position and mutual-visibility is already too vast to list in full. In the area of the general position problem, we highlight the following early papers~\cite{AnaChaChaKlaTho, Ghorbani-2021, Klavzar-2019, Patkos-2019, yao-2022} and the following recent ones~\cite{Araujo-2025, irsic-2024, KlaKriTuiYer-2023, KorzeVesel-2023, Kruft-2024, Roy-2025, Thomas-2024a, powers}, and in the area of the mutual-visibility, we highlight the following papers~\cite{BresarYero-2024, cicerone-2024b, Korze-2024, Kuziak-2023, Roy-2025, Tian-2024a}. 

In addition to the general topicality of the field, the research in this paper has two specific motivations. First, in the seminal paper~\cite{manuel-2018}, the general position number was determined for glued binary trees. Here we build on this result by determining all the other seven related invariants of glued binary trees. Second,  in~\cite[Theorem 2.1]{KlaPatRusYero-2021}, the gp-sets of the Cartesian product of two paths were enumerated. To our knowledge, this is so far the only (nontrivial) enumeration result in the area. Here we add to this the enumeration results for all eight invariants studied on glued binary trees. 

We proceed as follows. In the rest of this section, we give the other necessary definitions and call up inequality chains we need in the following. In Section~\ref{sec:glued}, we determine the general position and the mutual-visibility invariants of glued binary trees and also determine the number of the corresponding extremal sets. At the end of the section the results are extended to glued $t$-ary trees. In Section~\ref{sec:generalized}, we discuss generalized glued binary trees and extend some of the obtained results to this more general context. 
	
If $X\subseteq V(G)$, then the subgraph of $G$ induced by $X$ is denoted by $G[X]$. The open neighborhood of a vertex $u$ will be denoted by $N(u)$ and its degree by $\deg(u)$. For vertices $u$ and $v$ of $G$, the length of a shortest $u,v$-path is called {\em distance} and it is denoted by $d_G(u,v)$. A subgraph $H$ of $G$ is {\em isometric}, if $d_H(u,v) = d_G(u,v)$ for every two vertices $u$ and $v$ of $H$, and $H$ is {\em convex}, if for any vertices $u,v\in V(H)$, any shortest $u.v$-path in $G$ lies completely in $H$.   

If $G$ is a graph and $\tau \in \{\mu, \muo, \mud, \mut, \gp, \gpo, \gpd, \gpt\}$, then the number of $\tau$-sets of $G$ will be denoted by $\#\tau(G)$.

To conclude the introduction, we state the following sequences of inequalities which follow directly from the definitions and will be needed throughout. If $G$ is a graph, then:  
\begin{align}
\gpt(G) & \le \mut(G) \le \min\{\muo(G), \mud(G)\} \le \max\{\muo(G), \mud(G)\} \le \mu(G) \label{eq:1}\\
\gpt(G) & \le \min\{\gpo(G), \gpd(G)\} \le \max\{\gpo(G), \gpd(G)\} \le \gp(G) \le \mu(G) \label{eq:3}\\
\gpo(G) & \le \muo(G) \label{eq:5}\\
\gpd(G) & \le \mud(G) \label{eq:6} 
\end{align}

\section{Glued $t$-ary trees}
\label{sec:glued}

In this section we determine the four mutual-visibility invariants and the four general position invariants for all glued $t$-ary trees, where $t\ge 2$. The arguments for the general case are parallel to the arguments for the case $t=2$. From this reason, and not to introduce too much notation unnecessarily, we will elaborate the proofs for the case $t=2$. 

A {\em perfect t-ary tree} $T_{r,t}$, of depth $r\ge 1$, is a rooted tree in which all non-leaf vertices have $t$ children, and all leaves have depth $r$. In particular, $T_{1,t}\cong K_{1,t}$. The perfect $t$-ary tree $T_{r,t}$ has $\frac{t^{r+1}-1}{t-1}$ vertices and $t^r$ leaves. For $t\ge 2$, and $r\ge 1$, a {\em glued t-ary tree} $GT(r,t)$ is obtained from two copies of $T_{r,t}$ by pairwise identifying their leaves. The vertices obtained by identification are called {\em quasi-leaves} of $GT(r,t)$, the set of the quasi-leaves of $GT(r,t)$ will be denoted by $L(GT(r,t))$. If $u$ and $v$ are quasi-leaves with $N(u) = N(v)$, then we say that $u$ and $v$ are {\em twin quasi-leaves}. Note that if $r\ge 2$, then a vertex $u$ of $GT(r,t)$ is a quasi-leaf if and only if $\deg(u) = 2$ and $u$ lies in a $C_4$. 

Let $u$ be a vertex of $GT(r,t)$. Then we will denote by $T_u$ the tree rooted in $u$ and having as the leaves all the descendants of $u$ that are quasi-leaves. We say that vertices $u$ and $v$ of $GT(r,t)$ are {\em quasi-twins}, if  $T_u$ and $T_v$ have the same leaves. 

Perfect 2-ary trees and glued $2$-ary trees are respectively called {\em perfect binary trees} and {\em glued binary trees}. We will simplify the notation by setting $GT(r) = GT(r,2)$. See Fig.~\ref{fig:examples-for-GT(3)} for $GT(3)$. We will denote the two copies of $T_{r,2}$, from which $GT(r)$ is constructed, by $T_r^{(1)}$ and $T_r^{(2)}$, so that $L(GT(r)) = V(T_r^{(1)})\cap V(T_r^{(2)})$. Moreover, we set $V_r^{(1)} = V(T_r^{(1)}) \setminus L(GT(r))$ and $V_r^{(2)} = V(T_r^{(2)}) \setminus L(GT(r))$. Note that the mapping  $V(T_r^{(1)}) \rightarrow V(T_r^{(2)})$ which maps each vertex of $V_r^{(1)}$ to its quasi-twin in $V_r^{(2)}$, and maps each quasi-leaf to itself, is an isomorphism between $T_r^{(1)}$ and $T_r^{(2)}$.

The following observation will be used implicitly or explicitly throughout this section. Its proof follows by the construction of $GT(r)$. 

\begin{lemma}
\label{lem:Vr-is-convex}
If $r\ge 2$, then $GT(r)[T_r^{(i)}]$, $i\in [2]$, is an isometric subgraph of $GT(r)$. In addition, if $u,v\in V(T_r^{(i)})$, then there exist exactly two shortest $u,v$-paths if $u$ and $v$ are quasi-leaves, otherwise the shortest $u,v$-path is unique.  
\end{lemma}

\subsection{Mutual-visibility in glued binary trees}
\label{sec:mv-binary}
  
In this subsection we determine the mutual-visibility invariants of $GT(r)$ and begin with the mutual-visibility itself. 

\begin{lemma}\label{mu-quasipairs}
Let $S$ be a mutual visibility set of $GT(r)$. If $|S \cap V_r^{(1)}| \ge 2$, then corresponding to each vertex $ v \in S \cap V_r^{(1)}$,  we can find a pair of twin quasi-leaves in $T_v$ that are not in $S$ such that these pairs are pairwise disjoint for all vertices from $S\cap V_r^{(1)}$.
\end{lemma}

\begin{proof}
Let $S$ be a mutual visibility set of $GT(r)$ such that $|S \cap V_r^{(1)}| \ge 2$. 

If there exist two vertices $u,v \in S \cap V_r^{(1)}$ such that $u \in V(T_v)$, then let $w$ be the descendant of $v$ which is not on the shortest $u,v$-path. Now, $S \cap V_r^{(1)} \subseteq V(T_v) \setminus V(T_w)$. Moreover, for any two vertices $x,y \in (S \cap V_r^{(1)}) \setminus \{v\}$, the rooted trees $T_x$ and $T_y$ are disjoint, since otherwise, one of $x$ and $y$ is not $S$-visible to $v$. Also, the quasi-leaves of $T_x$ (as well as $T_y$) are not in $S$, as they are not $S$-visible to $v$ and the quasi-leaves of $T_w$ are not in $S$ as they are not $S$-visible to $u$. Therefore, any vertex $x \in S \cap V_r^{(1)} \setminus \{v\}$ can be assigned a pair of twin quasi-leaves from $V(T_x)$ and $v$ can be assigned a pair of twin quasi-leaves from $V(T_w)$. 

In the other case, for every $x,y \in S \cap V_r^{(1)}$, the rooted trees $T_x$ and $T_y$ are disjoint and the quasi-leaves of $T_x$ cannot be present in $S$ as they are not $S$-visible to $y$. Then again, any vertex $x \in S \cap V_r^{(1)}$ can be assigned a pair of twin quasi-leaves from $V(T_x)$.

In either case, to every vertex $v \in S \cap V_r^{(1)}$ we can assign a pair of twin quasi-leaves in $T_v$ that are not in $S$ such that these pairs are pairwise disjoint.
\end{proof}

\begin{theorem}\label{thm:mu-glued}
If $r\geq 2$, then 
$$\mu (GT(r)) = 2^r + 1\quad {\rm and}\quad \#\mu (GT(r)) = 2^{r+1} - 2\,.$$
\end{theorem}

\begin{proof}
Let $v$ be an arbitrary vertex from $V_r^{(1)}\cup V_r^{(2)}$, we may assume without loss of generality that $v\in V_r^{(1)}$. Set $S = L(GT(r)) \cup \{v\}$. Then any two quasi-leaves from $S$ are $S$-visible because $GT(r)[T_r^{(2)}]$ is isometric by Lemma~\ref{lem:Vr-is-convex}. Moreover, an arbitrary quasi-leaf and the vertex $v$ are also $S$-visible by Lemma~\ref{lem:Vr-is-convex}. Hence $S$ is a mutual-visibility set of $GT(r)$ and consequently $\mu(GT(r)) \geq 2^r + 1$.
			
To prove that $\mu(GT(r)) \leq 2^r + 1$, suppose on the contrary that there exists a mutual-visibility set $S$ of $GT(r)$ with cardinality at least $2^r + 2$. 
			
Let $|S\cap V_r^{(1)}| = k_1$ and $|S\cap V_r^{(2)}| = k_2$. Without loss of generality, suppose $k_1 \geq k_2$. If $k_1 = k_2 = 1$, then $S$ contains all the quasi-leaves, but then the (unique) vertex of $S$ from $V_r^{(1)}$ and the (unique) vertex of $S$ from $V_r^{(2)}$ are not $S$-visible. It follows that $k_1\ge 2$. 			
			
Let $u$ be an arbitrary vertex from $S\cap V_r^{(1)}$. Recall that $T_u$ denotes the tree rooted in $u$ having as the leaves all the descendants of $u$ which are quasi-leaves. Since $k_1\ge 2$, by Lemma~\ref{mu-quasipairs} we can assign to each vertex in $S\cap V_r^{(1)}$ a pair of twin quasi-leaves which do not lie in $S$ such that these pairs are pairwise disjoint. This in turn implies that  
\begin{align*}
|S| & = k_1 + k_2 + |S\cap L| \\
& \le k_1 + k_1 + (2^r - 2k_1) \\
& = 2^r\,.
\end{align*}
By this contradiction we can conclude that $\mu(GT(r)) \leq 2^r + 1$.

We have thus proved the first formula of the theorem. To prove the second, we claim that each $\mu$-set of $GT(r)$ is of the form $L \cup \{v\}$, where $v\in V_r^{(1)}\cup V_r^{(2)}$. So let $S$ be an arbitrary  $\mu$-set of $GT(r)$. As already proved, $|S| = 2^r + 1$. Setting $k = |S\cap (V_r^{(1)}\cup V_r^{(2)})|$ we need to prove that $k = 1$. Set further $k_1 = |S\cap V_r^{(1)}|$, $k_2 = |S\cap V_r^{(2)}|$, and assume without loss of generality that $k_1\ge k_2$. 
					
If $k_1\ge 2$, then by the above argument we get that $|S|\le 2^r$, a contradiction. There is nothing to prove if $k_2 = 0$. Hence we are left with the case $k_1 = k_2 = 1$. Let $u \in S\cap V_r^{(1)}$ and $v \in S\cap V_r^{(2)}$. Consider now the set $D_1$ of quasi-leaves in $T_u$ and the set $D_2$ of quasi-leaves in $T_v$. If $D_1\cap D_2 \ne \emptyset$, then $|D_1\cap D_2|\ge 2$ and no vertex from $D_1\cap D_2$ lies in $S$. Hence $|S|\le 2^r$, which is not possible. To complete the argument we claim that the case $D_1\cap D_2 = \emptyset$ also cannot happen. Indeed, if this would be the case, then a vertex from $D_1$ would not be $S$-visible with a vertex from $D_2$. 

We have thus established that $\mu$-set of $GT(r)$ are of the form $L \cup \{v\}$, where $v\in V_r^{(1)}\cup V_r^{(2)}$. This means that 
$$\#\mu (GT(r)) = |V_r^{(1)}\cup V_r^{(2)}| = 2^{r+1} - 2\,,$$
hence the second formula. 
\end{proof}

\begin{theorem}
If $r \geq 2$, then 
$$\muo (GT(r)) = 2^{r}\quad {\rm and}\quad \#\muo (GT(r)) = 1\,.$$
\end{theorem}
		
\begin{proof}
From Theorem~\ref{thm:mu-glued} we know that $\mu$-sets of $GT(r)$ are of the form $S = L(GT(r)) \cup \{v\}$, where $v\in V_r^{(1)}\cup V_r^{(2)}$. Assume without loss of generality that $v\in V_r^{(1)}$. Then no pair of vertices $v, u$, where $u\in V_r^{(2)}$ is $S$-visible. Hence $S$ is not an outer mutual-visibility set. Using~\eqref{eq:1} we get
$$\muo (GT(r)) < \mu(GT(r)) = 2^{r}+1\,.$$ 
On the other hand, using Lemma~\ref{lem:Vr-is-convex} we infer that $L(GT(r))$ is an outer mutual-visibility set of $GT(r)$, so that $\muo (GT(r)) \ge 2^{r}$. We can conclude that $\muo (GT(r)) = 2^{r}$. 

To prove that $\#\muo (GT(r)) = 1$, we claim that $L(GT(r))$ is the unique $\muo$-set. Let $S$ be a $\muo$-set of $GT(r)$. Let $v \in V_r^{(1)}$ be a vertex of depth at most $r-2$. If $v \in S$, then let $x$ and $y$ be the children of $v$. No vertex from $V(T_x)$ is $S$-visible to a vertex in $V(T_y)$, hence they cannot be present in $S$. Also, $x$ and $y$ are not $S$-visible to any of the vertices in $V_r^{(1)} \setminus V(T_v)$ and hence any vertex in $V_r^{(1)} \setminus V(T_v)$  cannot be present in $S$. Therefore, $S \cap V(T_r^{(1)}) = \{v\}$ and hence, $|S| = 2^r$ if only if $V_r^{(2)} \subset S$. Clearly, this is not possible and hence a vertex of depth less than $r-2$ cannot be present in $S$. 

The remaining vertices which may be present in $S$ can be partitioned into $2^{r-1}$ sets, each of which contains a pair of twin quasi-leaves and their neighbors, so that each set induces a convex $C_4$. Since, $\muo(C_4) \leq 2$, at most two vertices from each of this set can be present in $S$. Now, if a vertex $v$ of depth $r-1$ is present in $S$, then none of the vertices from the $C_4$ containing this $v$ can be in $S$, so that $|S| \leq 2^r - 1$, which is a contradiction. Therefore, the only possibility is to choose the twin quasi-leaves from each $C_4$ so that, $S = L(GT(r))$.
\end{proof}
		
\begin{theorem}\label{thm:mud-of-glued-binary-tree}
If $r\geq 2$, then 
$$\mut (GT(r)) = \mud (GT(r)) = 2^{r-1}\quad {\rm and}\quad \#\mut (GT(r)) = \#\mud (GT(r)) = 2^{2^{r-1}}\,.$$
\end{theorem}
		
\begin{proof}
Set $L = L(GT(r))$. Let $L = \{v_1, v_2, \ldots, v_{2^r}\}$, where $v_{2k-1}$ and $v_{2k}$ are twin quasi-leaves for $k \in [2^{r-1}]$. From Lemma~\ref{lem:Vr-is-convex}, we infer that $\{v_1, v_3, \ldots, v_{2^r-1}\}$ is a total mutual-visibility set of $GT(r)$, so that $\mut(GT(r)) \ge 2^{r-1}$. In view of~\eqref{eq:1}, the proof of the first formula will be completed by proving that $\mud(GT(r)) \le 2^{r-1}$. 
			
Let $S$ be an arbitrary $\mud$-set of $GT(r)$. We first claim that $S \cap V_r^{(1)} = \emptyset$ and suppose on the contrary that there exists a vertex $u\in S \cap V_r^{(1)}$. Assume first that $u$ is not the root of $T_r^{(1)}$. Let $u'$ and $u''$ be the children of $u$ and let $w$ be the parent of $u$. Then $S\cap \{u',u''\}\ne \emptyset$, for otherwise $u'$ and $u''$ are not $S$-visible. Assume without loss of generality that $u'\in S$. Further, $S\cap \{u'',w\}\ne \emptyset$, for otherwise $u''$ and $w$ are not $S$-visible. But now if $u''\in S$, then $u'\in S$ and $u''\in S$ are not $S$-visible, and if $w\in S$, then $u'\in S$ and $w\in S$ are not $S$-visible. Assuming that $u$ is the root of $T_r^{(1)}$, we arrive to a contradiction using a similar argument. 
			
We have thus proved that $S \cap V_r^{(1)} = \emptyset$. By the symmetry of $GT(r)$, we also have $S \cap V_r^{(2)} = \emptyset$. From this it readily follows that for any $k \in [2^{r-1}]$, at most one of the vertices $v_{2k-1}$ and $v_{2k}$ can belong to $S$, that is, at most one vertex of twin quasi-leaves can be in $S$. We can conclude that $\mud(GT(r)) = |S| \le 2^{r-1}$. This proves the first formula. 

To derive the second formula, the above arguments imply that each $\mud$-set and each $\mut$-set contains exactly one vertex from an arbitrary pair of twin quasi-leaves. As there are $2^{r-1}$ such pairs, the result follows. 
\end{proof}

In Fig.~\ref{fig:examples-for-GT(3)}, examples of a $\mu$-set, $\muo$-set (unique), $\mut$-set and $\mud$-set are shown. 

\begin{figure}[ht!]
\includegraphics[width=1\linewidth]{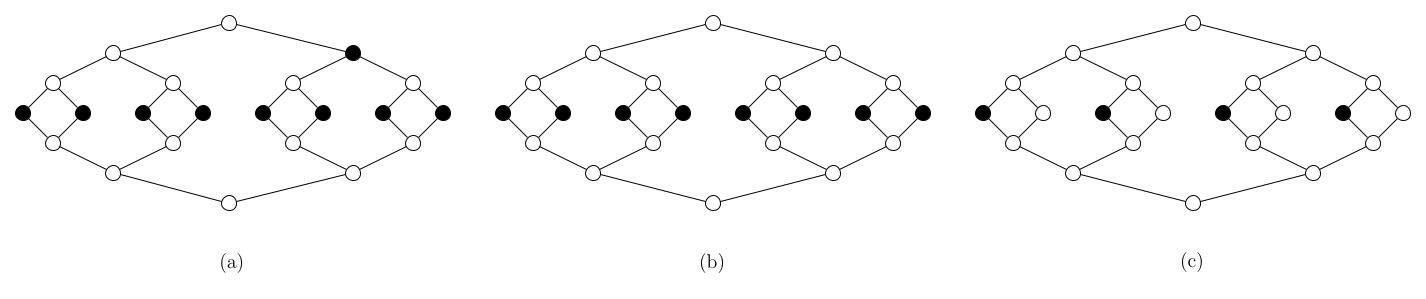}
\caption{$GT(3)$ and its (a) $\mu$-set (b) unique $\muo$-set (c) $\mut$-set, $\mud$-set}
\label{fig:examples-for-GT(3)}
\end{figure}

\subsection{General position in glued binary trees}
\label{sec:gp-binary}
  
In this subsection we determine the four general position invariants for glued binary trees. The general position number was earlier determined in~\cite[Proposition 3.8]{manuel-2018} by using the isometric path number of a graph. In the next theorem we reprove this result using an alternate argument. We further classify the gp-sets as well as the $\gpo$-sets. 

The next lemma follows directly from Lemma~\ref{mu-quasipairs}, since every general position set is a mutual-visibility set.

\begin{lemma}\label{gp-quasipairs}
Let $S$ be a general position set of $GT(r)$. If $|S \cap V_r^{(1)}| \ge 2$, then corresponding to each vertex $ v \in S \cap V_r^{(1)}$,  we can find a pair of quasi-leaves in $T_v$ that are not in $S$ such that these pairs are pairwise disjoint for all vertices from $S\cap V_r^{(1)}$.
\end{lemma}

\begin{theorem}\label{gp-gpo-glued binary tree}
If $r\geq 2$, then $\gpo (GT(r))  = \gp (GT(r)) = 2^r\,.$ Moreover,
$$\#\gpo(GT(r)) = 1 \quad {\rm and}\quad \#\gp(GT(r)) = 2^{r-1} + 1\,.$$
\end{theorem}
		
\begin{proof}
By Lemma~\ref{lem:Vr-is-convex}, $L(G(T))$ is an outer general position set of $G$ and hence $\gpo(G) \ge 2^r$. In view of~\eqref{eq:3}, the proof of the first two formulas will be completed by proving that $\gp(GT(r)) \le 2^r$.
			
Consider an arbitrary general position set $S$ of $GT(r)$. If $S \subseteq V(T_r^{(i)})$, for some $i \in [2]$, then $S$ is also a general position set of $T_r^{(i)}$ and hence it follows that $|S| \leq 2^r$. Therefore, assume that $S \cap V_r^{(i)}$ is non-empty and let $|S \cap V_r^{(i)}| = t_i$, $i \in [2]$. Assume without loss of generality that $t_1\ge t_2$. Now, by Lemma~\ref{gp-quasipairs}, corresponding to each vertex in $S \cap V_r^{(1)}$, we can find a pair of quasi-leaves that are not in $S$, such that these pairs are pairwise disjoint. Therefore,
$$|S| \leq 2^r + t_1 + t_2 - 2t_1 \leq 2^r + t_2 - t_1 \le 2^r\,.$$
Hence, $\gp(GT(R)) = 2^r$, thus establishing the first two formulas. 

By the above argument, $|S| = 2^r$ only when $t_1 = t_2$, hence let $t = t_1 = t_2$. Using Lemma~\ref{gp-quasipairs}, assign a pair of quasi-leaves from $T_u$ to each vertex $u \in S \cap V_r^{(1)} $ which are pair-wise disjoint and assign a pair of quasi-leaves from $T_v$ to each vertex $ v \in S \cap V_r^{(2)} $ which are also pair-wise disjoint. Note that $|S| = 2^r$ if and only if the quasi-leaves which are not present in $S$ are exactly the pairs which are assigned to the vertices of $S \cap V_r^{(i)} $, $i \in [2]$ and every pair of quasi-leaves assigned to a vertex in $S \cap V_r^{(1)}$ should also be assigned to a vertex in $S \cap V_r^{(2)}$. Let $x \in S \cap V_r^{(1)}$ and let $y \in S \cap V_r^{(2)}$ be such that the quasi-leaves assigned to $x$ and $y$ are the same. Without loss of generality assume that $y$ is of depth less than or equal to that of $x$. Then every quasi-leaf in $T_x$ is an internal vertex of a shortest $x,y$-path. Hence, none of them can be present in $S$. If $x$ is of depth less than $r-1$, then there are more pairs of quasi-leaves in $T_x$ than which are assigned to $x$ and these quasi-leaves must be assigned to some vertices in $S \cap V_r^{(1)}$, as well as to some vertices in $S \cap V_r^{(2)}$. But, if a vertex $z \in S \cap V_r^{(2)}$ of depth less than that of $y$ is assigned a pair of these quasi-leaves in $T_x$, then $y$ will be an internal vertex of a shortest $x,z$-path and if a vertex $z \in S \cap V_r^{(2)}$ of depth greater than that of $y$ is assigned a pair of these quasi-leaves in $T_x$, then $z$ will be an internal vertex of a shortest $x,y$-path, both of which contradicts $S$ is a general position set. Hence, $x$ must be of depth $r-1$. If $y$ is also of depth $r-1$, then $x$ and $y$ are quasi-twins and we cannot have any more vertices in $S \cap (V_r^{(1)} \cup V_r^{(2)})$ and hence, $t=1$. 

Now, if $y$ is of depth less than $r-1$ then $t > 1$. Since, $t > 1$, there is at least one more vertex $z \in S \cap V_r^{(2)}$. Also, $y$ is an internal vertex of a shortest $x,v$-path for every $v \in V_r^{(2)} \setminus V(T_y)$. Therefore, $z \in V(T_y)$ and $z$ should not be an internal vertex of a shortest $x,y$-path. Now, let $w \in S \cap V_r^{(1)}$ be the vertex which is assigned the same quasi-leaves as that of $z$. If $w$ is of depth less than or equal to that of $y$, then $x$ is an internal vertex of a shortest $w,y$-path and if $w$ is of depth greater than $y$, then $z$ is an internal vertex of a shortest $w,y$-path. In either case $S$ is not a general position set. Therefore $t > 1$ is not possible.

We have thus proved that a gp-set of $GT(r)$ is either $L(GT(r))$, or has the form 
$$(L(GT(r)) \setminus \{u,v\}) \cup N(u)\,,$$ 
where $u$ and $v$ are twin quasi-leaves. Since, there are $2^{r-1}$ twin quasi-leaves in $GT(r)$, it follows that $\#\gp(GT(r)) = 2^{r-1}+1$.

Note that $(L(GT(r)) \setminus \{u,v\}) \cup N(u)$ is not an outer general position set, since $u$ and $v$ are not in $S$-position with the remaining quasi-leaves of $GT(r)$. Therefore, the only $\gpo$-set is $L(GT(r))$ and hence $\#\gpo(GT(r)) = 1$.
\end{proof}

In Fig.~\ref{fig:examples-gpo-gp-for-GT(3)}, the unique $\gpo$-set and example of a gp-set are shown.

\begin{figure}[ht!]
            \centering 
            \includegraphics[width=0.7\linewidth]{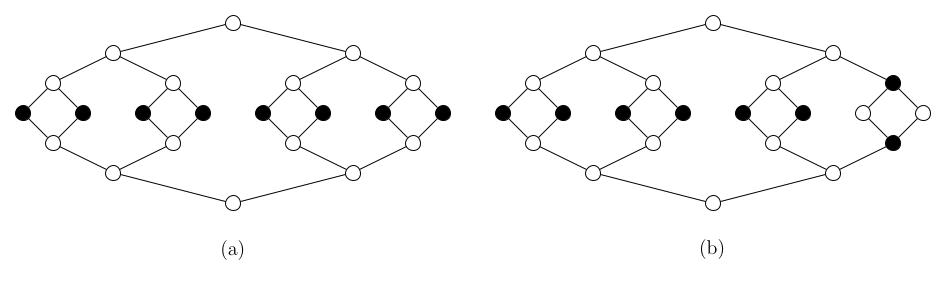}
            \caption{$GT(3)$ and its (a) unique $\gpo$-set and (b) a gp-set}
            \label{fig:examples-gpo-gp-for-GT(3)}
            \end{figure} 

\begin{lemma}\label{dual:K1,3}
    If $S$ is a dual general position set of $G$ then the central vertex of an induced $K_{1,3}$ cannot be in $S$.
\end{lemma}
\begin{proof}
    Let $S$ be a dual general position set of $G$ and let $W = \{v_1,v_2,v_3\}$ be the set of pendent vertices of an induced $K_{1,3}$ in $G$. By the pigeonhole principle, there exists a pair of vertices in $W$ which are either both present in $S$ or both not present in $S$. Since this pair of vertices must be in dual position, the central vertex of $K_{1,3}$ cannot be present in $S$. 
\end{proof}

\begin{corollary}
\label{cor:triangle-free}
If $G$ is a triangle-free graph with $\delta(G) \geq 3$, then $\gpd(G) = 0$. 
\end{corollary}

Triangle-free graph have girth at least $4$. For graphs $G$ of girth at least $6$ it was earlier proved in~\cite[Corollary 3.4]{TianKlavzar-2025} that $\gpd(G) = 0$ if and only if $\delta(G)\ge 2$.

\begin{theorem}
If $r\ge 2$, then $\gpt(GT(r)) = \gpd(GT(r)) = 0$.    
\end{theorem}

\begin{proof}
In~\cite[Theorem~2.1]{TianKlavzar-2025} it was proved that if $G$ is a connected graph and $X\subseteq V(G)$, then $X$ is a total general position set of $G$ if and only if it consists of (some of the) simplicial vertices of $G$. (Recall that a vertex is {\em simplicial} if its neighborhood induces a complete graph.) Since the glued binary trees $GT(r)$ do not contain simplicial vertices, we can conclude that $\gpt(GT(r)) = 0$. 

Now, let $S$ be an arbitrary dual general position set of $GT(r)$. By Lemma~\ref{dual:K1,3}, none of the vertices other than the quasi-leaves and the root vertices of binary trees can be present in $S$. Now, since the two children, say $u$ and $w$, of a root vertex $v$ are not present in $S$, $u$ and $w$ are in dual position implies that $v$ cannot be present in $S$. Similarly, if $v$ is a quasi-leaf and $u$ and $w$ are two parents of $v$ from two different copies of the binary tree, then again $u$ and $w$ (which are not in $S$) are in dual position implies that $v$ cannot be in $S$. Therefore, $S$ is empty. Hence, $\gpd(GT(r)) = 0$.
\end{proof}

The results of Sections~\ref{sec:mv-binary} and~\ref{sec:gp-binary} are summarized in Table~\ref{table:all-values}.

\begin{table}[ht!]
\begin{center}
\begin{tabular}{ ||c|c||c|c||c|c||c|c|| } 
\hline
\hline
$\mu$ & $2^r+1$ & $\#\mu$ & $2^{r+1} - 2$ & $\gp$ & $2^r$ & $\#\gp$ & $2^{r-1} + 1$ 
\phantom{$\begin{bmatrix} X \\ Y \end{bmatrix}$} \hspace*{-10mm}\\ 
\hline
$\muo$ & $2^r$ & $\#\muo$ & $1$ & $\gpo$ & $2^r$ & $\#\gpo$ & $1$ 
\phantom{$\begin{bmatrix} X \\ Y \end{bmatrix}$} \hspace*{-10mm} \\ 
\hline
$\mud$ & $2^{r-1}$ & $\#\mud$ & $2^{2^{r-1}}$ & $\gpd$ & $0$ & $\#\gpd$ & $-$
\phantom{$\begin{bmatrix} X \\ Y \end{bmatrix}$} \hspace*{-10mm} \\ 
\hline
$\mut$ & $2^{r-1}$ & $\#\mut$ & $2^{2^{r-1}}$ & $\gpt$ & $0$ & $\#\gpt$ & $-$ 
\phantom{$\begin{bmatrix} X \\ Y \end{bmatrix}$} \hspace*{-10mm} \\ 
\hline
\hline
\end{tabular}
\end{center}
\caption{Mutual-visibility and general position invariants in glued binary trees}
\label{table:all-values}
\end{table}

\subsection{Extension to glued $t$-ary trees}

The results of Sections~\ref{sec:mv-binary} and~\ref{sec:gp-binary} can be naturally extended to glued $t$-ary trees for any $t\ge 2$. The relevant proofs run in parallel with the proofs produced, so only the result is quoted here. 

\begin{theorem}
If $r\ge 2$ and $t\ge 2$, then 
\begin{enumerate}
    \item $\mu(GT(r,t)) = t^r + 1$ and $\# \mu(GT(r,t)) = t^{r+1} - 2$.
    \item $\muo(GT(r,t)) = t^r$ and $\# \muo(GT(r,t)) = 1$.
    \item If $\tau \in \{\mut, \mud\}$, then $\tau(GT(r,t)) = t^{r-1}(t - 1)$ and $\# \tau(GT(r,t)) = t^{t^{r-1}}$.
    \item $\gpo(GT(r,t)) = t^r$ and $\# \gpo(GT(r,t)) = 1$.
    \item $\gp(GT(r,t)) = t^r$ and $\# \gp(GT(r,t)) = t^{r-1} +1$.
    \item $\gpt(GT(r,t)) = \gpd(GT(r,t)) = 0$.
\end{enumerate}
\end{theorem}

\section{On generalized glued binary trees}
\label{sec:generalized}

In this paper we found the exact value of four variants each of mutual visibility and general position for glued binary trees. Moreover, we have also enumerated the number of $\tau$-sets for $\tau \in \{\mu, \muo, \mud, \mut, \gp, \gpo, \gpd, \gpt\}$.

Furthermore, we have seen that the results obtained can be directly extended to glued $t$-ary trees for any $t\ge 2$. Now, the concept of glued $t$-ary trees can be further generalized by gluing $n$ perfect $t$-ary trees, instead of two. More precisely, the $n^{\rm th}$ generalized glued binary tree $GT^{(n)}_r$ of depth $r$ is obtained from $n$ copies of $GT(r)$ by identifying their leaves. These copies will be denoted by $T_r^{(i)}$, $i\in [n]$, and the vertices obtained by the identification are again called {\em quasi-leaves} of $GT^{(n)}_r$. The set of the quasi-leaves will be denoted by $L(GT^{(n)}_r)$. Also, $GT^{(i,j)}_r$ denotes the subgraph induced by $V(T_r^{(i)})\cup V(T_r^{(j)})$. Note that $GT^{(i,j)}_r \cong GT(r)$. Finally, for $i\in [n]$ set $V_r^{(i)} = V(T_r^{(i)}) \setminus L(GT^{(n)}_r)$. 

Below, we extend some of the previous results to this generalized situation. First, for the total mutual-visibility, dual mutual-visibility, total general position and the total mutual-visibility, proofs are similar to those earlier, hence we can state: 

\begin{proposition}
If $r\ge 2$ and $n \ge 3$, then 
\begin{enumerate}
    \item $\mut(GT^{(n)}_r) = \mud(GT^{(n)}_r) = 2^{r-1}$ and $\# \mut(GT^{(n)}_r) = \# \mud(GT^{(n)}_r) = 2^{2^{r-1}}$, and
    \item $\gpt(GT^{(n)}_r) = \gpd(GT^{(n)}_r) = 0$.
\end{enumerate}
\end{proposition}

In the case of the outer general position number and the outer mutual-visibility number, we have the following result.

\begin{theorem}\label{thm:gpo-muo-generalized-glued-binary-tree}
If $r \geq 2$ and $n \geq 3$, then $\gpo(GT^{(n)}_r) = \muo(GT^{(n)}_r) = 2^r$ and $\# \gpo(GT^{(n)}_r) = \# \muo(GT^{(n)}_r) = 1$.
\end{theorem} 
\begin{proof}
By Lemma~\ref{lem:Vr-is-convex}, $L(GT^{(n)}_r)$ is an outer general position set of $GT^{(n)}_r$. Consequently, $\gpo(GT^{(n)}_r) \geq 2^r$. In view of~\eqref{eq:5}, the proof of $\gpo(GT^{(n)}_r) = \muo(GT^{(n)}_r) = 2^r$ will be completed by proving that $\muo(GT^{(n)}_r) \le 2^r$. Then since every outer general position set is an outer mutual-visibility set, the proof of $\# \gpo(GT^{(n)}_r) = \# \muo(GT^{(n)}_r) = 1$ will be completed by proving that $L(GT^{(n)}_r)$ is the unique $\muo$-set of $GT^{(n)}_r$.

Set $V = V(GT^{(n)}_r)$ and $L = L(GT^{(n)}_r)$. We claim that if $S$ is a $\muo$-set of $GT^{(n)}_r$ then $|S \cap (V\setminus L)| = 0$. If there exists $x, y \in S \cap V_r^{(i)}$ then $x$ is not $S$-visible to the descendants of $y$, which is not possible. Thus $|S \cap V_r^{(i)}| \leq 1$, for $i \in [n]$.

Now, if $x \in S \cap V_r^{(i)}$ then none of the descendants of $x$ including quasi-leaves can be in $S$. In addition, if $x \in S \cap V_r^{(i)}$ and $y \in S \cap V_r^{(j)}$ are such that $x$ and $y$ have a common pair of twin quasi-leaves as descendants then without loss of generality assume $x$ has depth greater than or equal to that of $y$ in which case, there exists a vertex $z \in V_r^{(j)}$ such that $x$ and $z$ are not $S$-visible. Hence we can assign to each vertex in $S \cap (V \setminus L)$ a pair of twin quasi-leaves which is not in $S$ such that these pairs are pairwise disjoint. Thus if $|S \cap (V \setminus L)| = k$ then $|S| \leq 2^r - 2k + k = 2^r - k$. Since $|S| \ge 2^r$, this implies that $|S \cap (V \setminus L)| = 0$. We have thus proved that $L$ is the unique $\muo$-set of $GT^{(n)}_r$ and hence we are done.  
\end{proof}

\begin{figure}[h]
\centering
\includegraphics[width=0.3\linewidth]{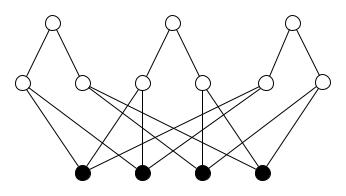}
\caption{$\gpo$-set and $\muo$-set of $GT_2^{(3)}$}
\label{fig:gt23}
\end{figure}

Moving on to the general position problem in generalized glued binary trees, it is a direct observation that in any maximal general position set quasi-leaves appear as twins, that is, if a quasi-leaf appears in a maximal general position set, then its twin quasi-leaf will also be there in $S$. 
Now, as in the case of general position set in glued binary trees, $(L(GT_r^{(n)}) \setminus \{u,v\}) \cup N(u)$ is a general position set of cardinality $2^r + n - 2$, so that $\gp(GT^{(n)}_r) \geq 2^r + n - 2$, for $n,r \geq 2$.

Though we do not have a proof, we have strong reasons to believe that $$\gp(GT^{(n)}_r) = 2^r + n - 2$$. For the time being, we leave it as an open problem.

If $n$ is large enough, then $\mu(GT^{(n)}_r)$ can be much larger than $\gp(GT^{(n)}_r)$. For instance, we can define a mutual-visibility set as follows. In each copy of the perfect binary tree, we have $2^{r-1}$ vertices of depth $r-1$. Therefore, we can choose $\binom{2^{r-1}}{2} = 2^{r-2}(2^{r-1}-1)$ distinct subsets of cardinality two. If $n \geq 2^{r-2} (2^{r-1}-1)$, then corresponding to each two element subset, choose two vertices of corresponding index from different copies of perfect binary trees. Clearly, this is a mutual-visibility set. Also, we can add one more vertex from each copy of the perfect binary tree from which vertices are not yet taken. The resultant set remains to be a mutual-visibility set and hence, if $n \geq 2^{r-2} (2^{r-1}-1)$, then 
$$\mu(GT^{(n)}_r) \geq 2(2^{r-2}(2^{r-1}-1)) + n - 2^{r-2}(2^{r-1}-1) = n + 2^{2r-3} - 2^{r-2}\,.$$
Still, if $n$ is small enough, then $\mu(GT^{(n)}_r)$ may come down to $2^r + n - 2$, which is the expected value of $\gp(GT^{(n)}_r)$ also.

\section*{Acknowledgments}
	
Dhanya Roy thank Cochin University of Science and Technology for providing financial support under University JRF Scheme. 
Sandi Klav\v zar was supported by the Slovenian Research Agency ARIS (research core funding P1-0297 and projects N1-0285, N1-0355). 

\section*{Declaration of interests}
 
The authors declare that they have no conflict of interest. 

\section*{Data availability}
 
Our manuscript has no associated data.


\baselineskip12pt	
\begin{thebibliography}{99}

\bibitem{AnaChaChaKlaTho}
B.S.~Anand, U.~Chandran S.V., M.~Changat, S.~Klav\v{z}ar, E.J.~Thomas,
Characterization of general position sets and its applications to cographs and bipartite graphs,
Appl.\ Math.\ Comput.\ 359 (2019) 84--89.

\bibitem{Araujo-2025}
J.~Araujo, M.C.~Dourado, F.~Protti, R.~Sampaio, The iteration time and the general position number in graph convexities, 
Appl.\ Math.\ Comput.\ 487 (2025) Paper 129084. 

\bibitem{BresarYero-2024}
B.~Bre\v{s}ar, I.G.~Yero,
Lower (total) mutual visibility in graphs,
Appl.\ Math.\ Comput.\ 465 (2024) Paper 128411.

\bibitem{chandran-2016} 
U.~Chandran S.V., G.J.~Parthasarathy, 
The geodesic irredundant sets in graphs, 
Int.\ J.\ Math.\ Combin.\ 4 (2016) 135--143.
  
\bibitem{cicerone-2023a} 
S.~Cicerone, G.~{Di Stefano}, L.~Dro\v{z}\dj ek, J.~Hed\v{z}et, S.~Klav\v{z}ar, I.G.~Yero,
Variety of mutual-visibility problems in graphs,
Theoret.\ Comput.\ Sci.\ 974 (2023) Paper 114096.

\bibitem{cicerone-2024b}
S.~Cicerone, G.~{Di Stefano}, S.~Klav\v{z}ar, I.G.~Yero,
Mutual-visibility problems on graphs of diameter two, 
European J.\ Combin.\ 120 (2024) Paper 103995.

\bibitem{cicerone-2024}
S.~Cicerone, G.~{Di Stefano}, S.~Klav\v{z}ar, I.G. Yero,
Mutual-visibility in strong products of graphs via total mutual-visibility,
Discrete Appl.\ Math.\ 358 (2024) 136--146.
  
\bibitem{distefano-2022}
G.~{Di Stefano},
Mutual visibility in graphs,
Appl.\ Math. Comput.\ 419 (2022) Paper 126850.

\bibitem{Ghorbani-2021}
M.~Ghorbani, H.R.~Maimani, M.~Momeni, F.R.~Mahid, S.~Klav\v{z}ar, G.~Rus,
The general position problem on Kneser graphs and on some graph operations,
Discuss.\ Math.\ Graph Theory 41 (2021) 1199--1213.

\bibitem{irsic-2024}
V.~Ir\v si\v c, S.~Klav\v zar, G.~Rus, J.~Tuite, 
General position polynomials,
Results Math.\ 79 (2024) Paper 110.

\bibitem{KlaKriTuiYer-2023} 
S.~Klav\v zar, A.~Krishnakumar, J.~Tuite, I.G.~Yero, 
Traversing a graph in general position, 
Bull.\ Aust.\ Math.\ Soc.\ (2023) 1--13.        

\bibitem{KlaPatRusYero-2021}
S.~Klav\v{z}ar, B.~Patk\'{o}s, G.~Rus, I.G.~Yero, 
On general position sets in {C}artesian products,
Results Math.\ 76 (2021) Paper 123.

\bibitem{Klavzar-2019}
S.~Klav\v{z}ar, I.G.~Yero,
The general position problem and strong resolving graphs,
Open Math.\ 17 (2019) 1126--1135.

\bibitem{Korner-1995}
J.~K\"orner,
On the extremal combinatorics of the Hamming space,
J.\ Comb.\ Theory Ser.\ A 71 (1995) 112--126.

\bibitem{KorzeVesel-2023}
D.~Kor\v{z}e, A.~Vesel, 
General position sets in two families of Cartesian product graphs,
Mediterr.\ J.\ Math.\ 20 (2023) Paper 203.

\bibitem{Korze-2024}
D.~Kor\v{z}e, A.~Vesel, 
Mutual-visibility sets in Cartesian products of paths and cycles,
Results Math.\ 79 (2024) Paper 116.

\bibitem{Kruft-2024}
E.~Kruft Welton, S.~Khudairi, J.~Tuite,
Lower general position in Cartesian products,
Commun.\ Comb.\ Optim.\  10 (2025) 110--125.

\bibitem{Kuziak-2023}
D.~Kuziak, J.A.~Rodr\'{\i}guez-Vel\'{a}zquez,
Total mutual-visibility in graphs with emphasis on lexicographic and Cartesian products,
Bull.\ Malays.\ Math.\ Sci.\ Soc.\ 46 (2023) Paper 197.

\bibitem{manuel-2018}
P.~Manuel, S.~Klav{\v z}ar,
A general position problem in graph theory,
Bull.\ Aust.\ Math.\ Soc.\ 98 (2018) 177--187.

\bibitem{Patkos-2019}
B.~Patk\'{o}s, 
On the general position problem on Kneser graphs,
Ars Math.\ Contemp.\ 18 (2020) 273--280.

\bibitem{Roy-2025}
D.~Roy, S.~Klav\v{z}ar, A.S.~Lakshmanan,
Mutual-visibility and general position in double
graphs and in Mycielskians,
Appl.\ Math.\ Comput.\ 488 (2025) Paper 129131.

\bibitem{Thomas-2024a}
E.J.~Thomas, U.~Chandran S.V., J.~Tuite, G.~Di Stefano,
On the general position number of Mycielskian graphs,
Discrete Appl.\ Math.\ 353 (2024) 29--43. 

\bibitem{Tian-2024a}
J.~Tian, S.~Klav\v{z}ar,
Graphs with total mutual-visibility number zero and total mutual-visibility in Cartesian products,
Discuss.\ Math.\ Graph Theory 44 (2024) 1277--1291. 

\bibitem{TianKlavzar-2025}
J.~Tian, S.~Klav\v{z}ar,
Variety of general position problems in graphs,
arXiv:2402.17338 [math.CO], 
also: Bull.\ Malay.\ Math.\ Soc., to appear.  

\bibitem{powers} 
J.~Tian, K.~Xu, 
On the general position number of the $k$-th power graphs, Quaest.\ Math.\ (2024) 1--16.

\bibitem{yao-2022}
Y.~Yao, M.~He, S.~Ji,
On the general position number of two classes of graphs,
Open Math.\ 20 (2022) 1021--1029.

\end{thebibliography}
\end{document}